\documentclass[a4,11pt]{amsart}
\usepackage[top=3cm, bottom=3cm, left=2.5cm, right=2.5cm]{geometry}
\usepackage{amssymb,amsmath,amsthm,txfonts, amsfonts}
\usepackage{amscd}
\usepackage[all]{xy} 
\usepackage[colorlinks=true, linkcolor=blue, citecolor=blue, urlcolor=blue]{hyperref}

\usepackage{tikz}
\usetikzlibrary{arrows.meta}

\usepackage{newtxtext}
\usepackage{newtxmath}

\usepackage{graphicx}
\usepackage{xcolor}

\usepackage{mathrsfs}
\usepackage{cite}
\usepackage{nicefrac}
\usepackage{textcomp}
\usepackage{protosem}

\usepackage{lscape}
\usepackage[hang,small,bf]{caption}
\usepackage[subrefformat=parens]{subcaption}
\newtheorem{thm}{Theorem}[section]
\newtheorem{lem}{Lemma}[section]

\newtheorem{prop}{Proposition}[section]
\theoremstyle{definition}
\newtheorem{defn}{Definition}[section]
\theoremstyle{remark}
\newtheorem{rem}{\textbf{Remark}}[section]

\newtheorem{ex}{\textbf{Example}}[section]
\newtheorem{q}{\textbf{Question}}[section]

 \makeatletter
    
    \@addtoreset{equation}{section}
  \makeatother

\makeatletter
      \def\@makefnmark{%
         \leavevmode
            \raise.9ex\hbox{\check@mathfonts
                \fontsize\sf@size\z@\normalfont%
                            \@thefnmark}%
       }
      \makeatother

\makeatletter
\@namedef{subjclassname@2020}{\textup{2020} Mathematics Subject Classification}
\makeatother

\begin{document}

\title{Minus one Homogeneous Euler Flows are Geodesible}
\author{Ken Abe, Naoki Sato, and Chunjing Xie}
\date{}
\address[Ken Abe]{Department of Mathematics, Graduate School of Science, Osaka Metropolitan University, 3-3-138 Sugimoto, Sumiyoshi-ku Osaka, 558-8585, Japan}
\email{kabe@omu.ac.jp}
\address[Naoki Sato]{National Institute for Fusion Science,  322-6 Oroshi-cho Toki-city, Gifu 509-5292, Japan}
\email{sato.naoki@nifs.ac.jp}
\address[Chunjing Xie]{School of Mathematical Sciences, MOE-LSC, CMA-Shanghai, Shanghai Jiao Tong University, 200240 Shanghai, P. R. China}
\email{cjxie@sjtu.edu.cn}

\dedicatory{Dedicated to the memory of Professor Dr. Hermann Sohr}

\date{\today}

\begin{abstract}
In this paper, we study $(-1)$-homogeneous steady solutions to the Euler equations on $\mathbb{R}^n \setminus \{0\}$. In low dimensions $n=2,3$, such flows are known to be essentially trivial. In contrast, we show that in higher dimensions $n \ge 4$, every $(-1)$-homogeneous Euler flow is a geodesible vector field with constant Bernoulli function. Moreover, any $(-1)$-homogeneous geodesible field is induced by a geodesible field on the sphere $\mathbb{S}^{n-1}$. In particular, in the case $n=4$, every $(-1)$-homogeneous Euler flow is obtained as an extension of a Beltrami field on $\mathbb{S}^{3}$. Our rigidity result contrasts with the flexibility result of Cardona (2021), which shows that on higher odd-dimensional Riemannian manifolds, a rich variety of Euler flows can be constructed from geodesible fields by exploiting the freedom in the choice of metric.
\end{abstract}

\maketitle


\tableofcontents

\section{Introduction}

\subsection{Flexibility of steady Euler flows}

In dimensions four and higher, the existence of steady solutions to the Euler equations remains largely unknown, particularly in the case of a flat metric. See \cite[II.6 \& A.II]{AK21} for a comprehensive reference. Apart from trivial extensions of lower-dimensional solutions, the only intrinsically high-dimensional examples known so far are the Hopf fields on odd-dimensional spheres. 

However, the work of Ghrist \cite{Ghrist}, together with more recent developments by Cardona \cite{C21}, reveals that the situation changes substantially when one allows flexibility in the choice of the Riemannian metric. In odd dimensions, it is possible to construct steady solutions with arbitrary topology within the class of flows with constant Bernoulli function. See Gonz\'{a}lez-Prieto et al. \cite{GPMPS25} for a recent survey.

Motivated by this perspective, we consider the steady Euler equations on an $n$-dimensional manifold $\Omega$ equipped with a Riemannian metric $g$ and associated volume form $\mu$:
\begin{equation}
\begin{aligned}
\iota_{\boldsymbol{u}} du + d\Pi &= 0, \\
d \iota_{\boldsymbol{u}} \mu &= 0.
\end{aligned}
\label{eq: E}
\end{equation}
Here, $\boldsymbol{u}$ denotes the velocity vector field and $\Pi$ the Bernoulli function. We denote by $u = g(\boldsymbol{u}, \cdot)$ the associated one-form and set the pressure by $\Pi=p+\frac{1}{2}g(\boldsymbol{u},\boldsymbol{u})$. The symbol $d$ stands for the exterior derivative, and $\iota_{\boldsymbol{u}}$ denotes contraction with the vector field $\boldsymbol{u}$. We say that $\boldsymbol{u}$ is volume-preserving if the second equation is satisfied.

There are three distinct notions associated with the vector field $\boldsymbol{u}$ in the context of the Euler equations \cite[\S 2]{C21}.

\begin{defn}[Eulerisable fields \cite{PSRT21}]
We say that a volume-preserving vector field $\boldsymbol{u}$ is \textit{Eulerisable} if there is a metric $g$ for which $\boldsymbol{u}$ satisfies the Euler equations \eqref{eq: E} for some Bernoulli function $\Pi$.
\end{defn}

\begin{defn}[Geodesible fields]
We say that a vector field $\boldsymbol{u}$ is \textit{geodesible} if there exists a metric $g$ such that its orbits are geodesics.
\end{defn}

\begin{defn}[Beltrami fields]
Let $(\Omega,g)$ be an odd-dimensional manifold of dimension $n=2m+1$. We say that $\boldsymbol{u}$ is a \textit{Beltrami field} if there exists a scalar function $\lambda$ such that $\textrm{curl}\ \boldsymbol{u} = \lambda \boldsymbol{u}$, where the vorticity field $\textrm{curl}\ \boldsymbol{u}$ is defined by
\[
\iota_{\textrm{curl}\ \boldsymbol{u}} \mu = (du)^m,
\]
with $u = g(\boldsymbol{u}, \cdot)$.
\end{defn}

These notions capture different geometric structures associated with steady Euler flows.

\begin{rem}
It is shown in \cite{PSRT21} that a non-vanishing, volume-preserving vector field $\boldsymbol{u}$ is Eulerisable if and only if there exists a one-form $u$ such that $u(\boldsymbol{u})>0$ and $\iota_{\boldsymbol{u}}u$ is exact.
\end{rem}
\begin{rem}
It is known \cite{Gluck}, \cite{C21} that a vector field $\boldsymbol{u}$ is geodesible if and only if there exists a one-form $u$ such that $u(\boldsymbol{u})>0$ and 
\begin{align}
\iota_{\boldsymbol{u}}du=0.
\end{align}
If $\boldsymbol{u}$ is geodesible, one can choose a metric $g$ so that $u=g(\boldsymbol{u},\cdot)$, e.g., \cite[Lemma 3.2]{C21}.
\end{rem}

Let $\boldsymbol{u}$ be a non-vanishing, volume-preserving vector field. Then the following are equivalent: $\boldsymbol{u}$ is Eulerisable with constant Bernoulli function, and $\boldsymbol{u}$ is geodesible. Moreover, any non-vanishing, volume-preserving geodesible vector field is a Beltrami field in odd dimensions. Indeed, if $\iota_{\boldsymbol{u}}du=0$, it follows that  
\begin{align*}
\iota_{\boldsymbol{u}\wedge \textrm{curl}\ \boldsymbol{u}}\mu=\iota_{\boldsymbol{u}}\iota_{\textrm{curl}\ \boldsymbol{u}}\mu=\iota_{\boldsymbol{u}}(du)^{m}=0.
\end{align*}
Hence $\boldsymbol{u}\wedge \textrm{curl}\ \boldsymbol{u}=0$ and $\boldsymbol{u}$ is a Beltrami field.

In dimension $n=3$, these notions coincide. However, in higher odd dimensions, they no longer agree because $\iota_{\boldsymbol{u}}(du)^{m}=0$ does not always imply $\iota_{\boldsymbol{u}}du=0$. In fact, Cardona \cite[Theorem 1.2]{C21}, \cite[Remark 3.19]{GPMPS25} showed that for $n=2m+1>3$ there exist volume-preserving Beltrami vector fields that are neither geodesible nor solutions to the Euler equations for any choice of Riemannian metric; see Figure \ref{f1}. It is known that the converse holds in higher odd dimensions under the additional condition $\lambda>0$ \cite[Proposition 3.18]{GPMPS25}.  

\begin{figure}[h]
\centering
\begin{tikzpicture}[scale=2, rotate=90]
  \draw (0.0,-0.20) circle [radius=1.36];

  \draw (-0.8,-0.8) rectangle (0.8,2.6);
  \draw (-0.8,0.9) -- (0.8,0.9);

\node[align=center] at (0,1.8) {(A) Eulerisable \\ with $d\Pi \neq 0$};
\node[align=center] at (0,0.05) {(B) Geodesible \\ (Eulerisable \\ with $d\Pi = 0$)};
  \node at (0.0,-2.2) {(C) Beltrami};
\end{tikzpicture}
\caption{Higher odd-dimensional (A) Eulerisable, (B) geodesible and (C) Beltrami fields in the non-vanishing and volume-preserving case}
\label{f1}
\end{figure}
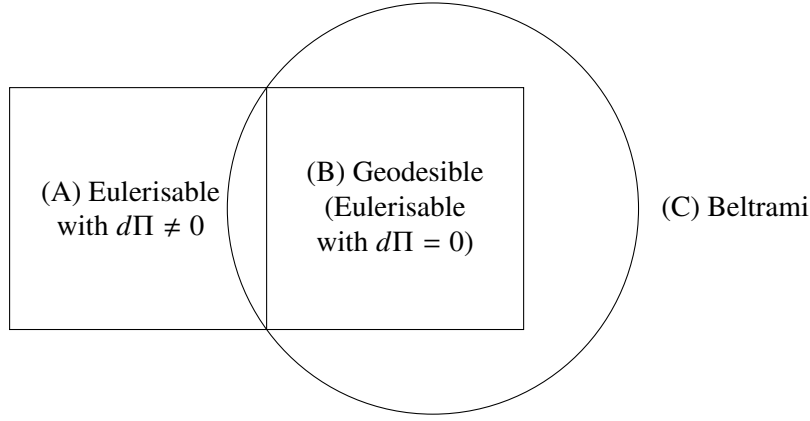

The flexibility viewpoint described above is closely related to recent developments 
on universality phenomena in fluid dynamics \cite{Tao18}, \cite{Tao20}. In particular, it has been shown that 
steady Euler flows can exhibit remarkable universality properties, allowing one to 
embed highly complex dynamics into solutions of the Euler equations 
\cite{Car23}, \cite{GPMPS25}.

\subsection{Rigidity of homogeneous Euler flows}

Following the flexibility–rigidity terminology \cite{CDG21b}, we investigate a rigidity phenomenon for higher-dimensional steady solutions of the Euler equations. In contrast to the work \cite{C21} on the flexibility of steady Euler flows on odd-dimensional Riemannian manifolds with freedom in the choice of metric, in this work we investigate the rigidity of steady solutions to the Euler equations in higher dimensions under the flat metric on $\mathbb{R}^n \setminus \{0\}$: 
\begin{equation}
\begin{aligned}
\iota_{\boldsymbol{u}}du+d \Pi&=0,\\
\textrm{div}\ {\boldsymbol{u}}&=0.
\end{aligned}
\label{eq: ER}
\end{equation}
Here we identify one-forms and vector fields via the flat metric, and denote by $u = \boldsymbol{u}^{\flat}$ the 1-form associated with the vector field $\boldsymbol{u}$. The first equation in \eqref{eq: ER} is the differential form formulation of 
$\boldsymbol{u} \cdot \nabla \boldsymbol{u} + \nabla p = 0$, where $\Pi = p + \frac{1}{2}|\boldsymbol{u}|^2$.

The Euler equations in the whole space are invariant under the scaling transformation
\[
\boldsymbol{u}(\boldsymbol{x}) \mapsto \boldsymbol{u}_{\lambda}(\boldsymbol{x}) = \lambda^{\alpha} \boldsymbol{u}(\lambda \boldsymbol{x})
\]
for any constant $\lambda > 0$ and $\alpha \in \mathbb{R}$. In this work, we investigate the rigidity of solutions exhibiting such scaling symmetry. We say that $\boldsymbol{u}$ is ($-\alpha$)-\textit{homogeneous} if there exists $\alpha\in \mathbb{R}$ such that $\boldsymbol{u}(\boldsymbol{x})= \lambda^{\alpha} \boldsymbol{u}(\lambda \boldsymbol{x})$ for all $\lambda>0$ and $\boldsymbol{x}\in \mathbb{R}^{n}\setminus \{0\}$. We say that $(\boldsymbol{u},p)$ is a ($-\alpha$)-homogeneous solution to \eqref{eq: ER} if ($-\alpha$)-homogeneous $\boldsymbol{u}$ and $(-2\alpha)$-homogeneous $p$ satisfy \eqref{eq: ER}. 

$(-\alpha)$-homogeneous solutions to the Euler equations in two and three dimensions have been studied by Luo and Shvydkoy \cite{LuoShvydkoy}, \cite{LS17}, Shvydkoy \cite{Shv}, and the first author \cite{Abe11}. 
They have also been investigated for the inviscid Boussinesq equations in \cite{AGJ} and for the generalized SQG equations in \cite{PC}, \cite{AGSJ}.

The existence and non-existence of $(-\alpha)$-homogeneous solutions depend on the value of $\alpha$. When $\alpha = 1$, the spherical component of $\boldsymbol{u}$ admits a first integral, and consequently a strong rigidity holds. Namely, \\

\begin{itemize}
\item For $n=2$, any ($-1$)-homogeneous solutions $(\boldsymbol{u},p)\in C^{1}(\mathbb{R}^{2}\setminus \{0\})$ to \eqref{eq: ER} are the irrotational solution
\begin{align}
\boldsymbol{u}=\frac{1}{\rho}(a\boldsymbol{e}_{\theta}+f\boldsymbol{e}_{\rho}),\quad p=\frac{1}{\rho^{2}}q,   \label{2Dhom}
\end{align}
for $a,f,q\in \mathbb{R}$, where $(\rho,\theta)$ are the polar coordinates and $\boldsymbol{e}_{\theta}={}^{t}(-\sin\theta,\cos\theta)$ and $\boldsymbol{e}_{\rho}={}^{t}(\cos\theta, \sin\theta)$.\\
\item For $n=3$, there exist no ($-1$)-homogeneous solutions $(\boldsymbol{u},p)\in C^{1}(\mathbb{R}^{3}\setminus \{0\})$ to \eqref{eq: ER} \cite[Proposition 2.1]{Shv} (see also \cite[Theorem 2.4]{Abe11}).
\end{itemize}

\subsection{The statement of the main results}

In contrast to the low-dimensional case, we show that in higher dimensions $(-1)$-homogeneous solutions on $\mathbb{R}^n\setminus\{0\}$ can be constructed from geodesible fields on $\mathbb{S}^{n-1}$. Moreover, all $(-1)$-homogeneous Euler solutions are necessarily geodesible. In what follows, we shall refer to a divergence-free vector field $\boldsymbol{u}$ as geodesible if it satisfies $\iota_{\boldsymbol{u}} du = 0$, even if $\boldsymbol{u}$ has zeros. The main result of this paper is as follows:

\begin{thm}\label{t:MT}
Let $n \ge 4$. Then every $(-1)$-homogeneous solution $(\boldsymbol{u},p)\in C^{1}(\mathbb{R}^{n}\setminus \{0\})$ to \eqref{eq: ER} is geodesible with respect to the flat metric. In particular, $\Pi=0$ and
\begin{align}
\boldsymbol{u}(\boldsymbol{x})=\frac{1}{\rho}\boldsymbol{v}(\boldsymbol{y}),\quad p(\boldsymbol{y})=-\frac{1}{2\rho^{2}}|\boldsymbol{v}(\boldsymbol{y})|^{2},
\label{eq:GF}
\end{align}
where $\rho=|\boldsymbol{x}|$ and $\boldsymbol{y}=\boldsymbol{x}/|\boldsymbol{x}|\in \mathbb{S}^{n-1}$. Here $\boldsymbol{v}$ is a geodesible vector field on $\mathbb{S}^{n-1}$ satisfying
\begin{equation}
\iota_{\boldsymbol{v}} dv=0,\quad
\operatorname{div}_{S}\boldsymbol{v}=0,
\label{eq:BS}
\end{equation}
where $v=\boldsymbol{v}^{\flat}$. In particular, in the case $n=4$, $\boldsymbol{v}$ is a Beltrami field on $\mathbb{S}^{3}$, that is,
\begin{align}
\operatorname{curl}_{S}\boldsymbol{v}=\lambda \boldsymbol{v},\quad 
\operatorname{div}_{S}\boldsymbol{v}=0,
\label{eq:Beltrami}
\end{align}
where $\lambda=\lambda(\boldsymbol{y})$ denotes the proportionality factor.
\end{thm}

Theorem \ref{t:MT} provides the \textit{rigidity} result for solutions to the higher-dimensional steady Euler equations under the flat metric: all $(-1)$-homogeneous Euler flows in $\mathbb{R}^n \setminus \{0\}$ equipped with the flat metric are necessarily geodesible. This result contrasts with Cardona's \textit{flexibility} result for steady Euler flows \cite{C21}, which shows that on any odd-dimensional Riemannian manifold there exist many geodesible fields, reflecting the freedom in the choice of the metric. This contrast is summarized in Table \ref{T1}.\\

\begin{table}[h]
\begin{tabular}{|c|c|c|}
\hline
                  & Cardona \cite[Theorem 1.1]{C21}                                                                                         & This work (Theorem \ref{t:MT})                                                                      \\ \hline
Geometric setting & Riemannian manifolds                                                                             & $\mathbb{R}^{n}\setminus\{0\}$                                                             \\ \hline
Metric            & Variable                                                                                         & Fixed                                                                                        \\ \hline
Dimensions            & $n=2m+1>3$                                                                                       &  $n\geq 4$                                                                                        \\ \hline
Result            & \begin{tabular}[c]{@{}c@{}}Existence of geodesible flows \\ with arbitrary topology\end{tabular} & \begin{tabular}[c]{@{}c@{}}All $(-1)$-homogeneous Euler flows\\  are geodesible\end{tabular} \\ \hline
Nature            & Flexibility                                                                                      & Rigidity                                                                                     \\ \hline
\end{tabular}
\caption{Flexibility versus rigidity}
\label{T1}
\end{table}

From Theorem \ref{t:MT}, it follows that a geodesible vector field on $\mathbb{S}^{n-1}$ does not give rise to a $(-1)$-homogeneous Euler solution in $\mathbb{R}^{n}\setminus\{0\}$; rather, only $(-1)$-homogeneous geodesible fields arise in this setting. 
When $\alpha \neq 1$, geodesible vector fields on the sphere are subject to a strong rigidity constraint: their speed must be constant on $\mathbb{S}^{n-1}$. In particular, in odd dimensions, no such $(-\alpha)$-homogeneous Euler solutions exist. 
By contrast, in even dimensions, such $(-\alpha)$-homogeneous Euler solutions can be constructed for any $\alpha \neq 0,\ 1,\, n-1$.

\begin{thm}\label{t:MT2}
Let $n \geq 4$, and let $\alpha \neq  1,\, n-1$. 
Let $\boldsymbol{v}$ be a geodesible vector field on $\mathbb{S}^{n-1}$. 
Define
\begin{align}
\boldsymbol{u}(\boldsymbol{x}) = \frac{1}{\rho^{\alpha}}\bigl(\boldsymbol{v}(\boldsymbol{y}) + f(\boldsymbol{y})\boldsymbol{y}\bigr), 
\qquad 
p(\boldsymbol{x}) = \frac{1}{\rho^{2\alpha}} q(\boldsymbol{y}).
\label{eq:GF2}
\end{align}
Then $(\boldsymbol{u},p)$ is a $(-\alpha)$-homogeneous solution to \eqref{eq: ER} if and only if $\alpha \neq 0$ and $n$ is even and $|\boldsymbol{v}|$ is constant. In this case, $f=0$, $ q = -\frac{1}{2\alpha}|\boldsymbol{v}|^{2}$, and the Bernoulli function is given by
\begin{align}
\Pi(\boldsymbol{x}) = \frac{1}{2}\left(1 - \frac{1}{\alpha}\right)\frac{1}{\rho^{2\alpha}}|\boldsymbol{v}|^{2}.
\label{eq:BH}
\end{align}
\end{thm}
Conversely, a geodesible vector field on $\mathbb{S}^{n-1}$ gives rise to a $(-\alpha)$-homogeneous geodesible field in $\mathbb{R}^{n}\setminus\{0\}$ only when $\alpha = 1$.

\begin{thm}\label{t:MT3}
Let $n \geq 4$. 
A geodesible vector field on $\mathbb{S}^{n-1}$ gives rise, via the extension \eqref{eq:GF2}, to a $(-\alpha)$-homogeneous geodesible solution of the Euler equations \eqref{eq: ER} in $\mathbb{R}^{n}\setminus \{0\}$ if and only if $\alpha = 1$. 
\end{thm}

\subsection{Implications of the main results}

Theorem~\ref{t:MT} establishes a rigidity theorem for $(-1)$-homogeneous Euler flows in $\mathbb{R}^n \setminus \{0\}$, showing that every such flow arises from a geodesible vector field on $\mathbb{S}^{n-1}$. On the other hand, Theorems~\ref{t:MT2} and \ref{t:MT3} provide necessary and sufficient conditions on the homogeneity exponent $\alpha$ under which a geodesible vector field on $\mathbb{S}^{n-1}$ can be extended to a $(-\alpha)$-homogeneous Euler flow or a $(-\alpha)$-geodesible vector field on $\mathbb{R}^n \setminus \{0\}$.

For the flat Euclidean metric, the existence and topological structure of genuinely higher-dimensional Euler flows on $\mathbb{R}^n$ remain largely unexplored. Theorems~\ref{t:MT}--\ref{t:MT3} shed light on this problem under the assumption of homogeneity. More specifically, their contributions can be highlighted through the following three aspects.\\

\begin{itemize}
\item \textbf{Rigidity}. We establish the first rigidity theorem for higher-dimensional Euler flows under the assumption of $(-1)$-homogeneity. More precisely, Theorem~\ref{t:MT} establishes a one-to-one correspondence between $(-1)$-homogeneous Euler flows in $\mathbb{R}^n \setminus \{0\}$ and geodesible vector fields on $\mathbb{S}^{n-1}$.\\

\item \textbf{Existence}. Theorem \ref{t:MT} also asserts that, in the flat Euclidean space $\mathbb{R}^n \setminus \{0\}$, including even-dimensional cases, solutions to the Euler equations can be constructed from geodesible vector fields on $\mathbb{S}^{n-1}$. In particular, when $n=4$, the abundance of Beltrami fields on $\mathbb{S}^{3}$ yields a large family of ($-1$)-homogeneous Euler flows on $\mathbb{R}^{4}\setminus \{0\}$. This shows that, even within the class arising from geodesible fields on $\mathbb{S}^{3}$
, one obtains a remarkably rich collection of four-dimensional Euler flows.\\

\item \textbf{Degree of homogeneity}. We give necessary and sufficient conditions for a geodesible vector field on $\mathbb{S}^{n-1}$ to extend to a $(-\alpha)$-homogeneous Euler flow or a $(-\alpha)$-geodesible vector field on $\mathbb{R}^n \setminus \{0\}$ (Theorems~\ref{t:MT2} and \ref{t:MT3}). As a consequence, we obtain a characterization of when a geodesible vector field on the sphere extends to a homogeneous Euler flow or a geodesible vector field in Euclidean space.
\end{itemize}

\subsection{Ideas of the proof}

The key ingredient in the proofs of Theorems \ref{t:MT}-\ref{t:MT3} is the system on a sphere $\mathbb{S}^{n-1}$. Following the case $n=3$ \cite{Shv}, \cite{Abe11}, we denote ($-\alpha$)-homogeneous solutions to \eqref{eq: ER} by 
\begin{align}
\boldsymbol{u}(\boldsymbol{x})
=
\frac{1}{\rho^{\alpha}}
\bigl(\boldsymbol{v}(\boldsymbol{y}) + f(\boldsymbol{y}) \boldsymbol{y}\bigr),
\qquad
p(\boldsymbol{x})
=
\frac{1}{\rho^{2\alpha}} q(\boldsymbol{y}),
\end{align}
for $\rho=|\boldsymbol{x}|$ and $\boldsymbol{y}=\boldsymbol{x}/|\boldsymbol{x}|$, and deduce a system for $(\boldsymbol{v},f,q)$ on $\mathbb{S}^{n-1}$: 
\begin{equation} 
\begin{aligned}
(1-\alpha) f\boldsymbol{v}+ \iota_{\boldsymbol{v}}d \boldsymbol{v}+\frac{1}{2}d h&=0,\\
\boldsymbol{v}\cdot \nabla_{S}f-|\boldsymbol{v}|^{2}-\alpha |f|^{2}-2\alpha q&=0,\\
\textrm{div}_S\ \boldsymbol{v}+(n-1-\alpha)f&=0,
\end{aligned}
\label{eq:HE0}
\end{equation}
for $h=2q+|\boldsymbol{v}|^{2}$. Here, the subscript $S$ indicates that the corresponding gradient and divergence are taken on $\mathbb{S}^{n-1}$. 

\subsubsection{Key ideas of the proof of Theorem \ref{t:MT}}
For $\alpha=1$, the first equation is nothing but the momentum equation of the Euler equations on $\mathbb{S}^{n-1}$ with the Bernoulli function $h/2$. Taking contraction with $\boldsymbol{v}$,
\begin{align}
\boldsymbol{v}\cdot \nabla_S h=0.   \label{eq:1stint}
\end{align}
Observe that the tangential vector field $\boldsymbol{v}$ is geodesible (i.e., $\iota_{\boldsymbol{v}}d\boldsymbol{v}=0$ and $\textrm{div}_{S}\ \boldsymbol{v}=0$) if $h=f=0$. Combining \eqref{eq:1stint} with the last two equations of \eqref{eq:HE0} for $\alpha=1$ (i.e., $\boldsymbol{v}\cdot \nabla_S f=f^{2}+h$ and $\textrm{div}_S\ \boldsymbol{v}+(n-2)f=0$), we deduce the key identity 
\begin{align}
\textrm{div}_S\ (\boldsymbol{v} f^{k}h^{l})
=(2-n+k)f^{k+1}h^{l}+kf^{k-1}h^{l+1},  \label{eq:Identity0}
\end{align}
for non-negative integers $k,l\in \mathbb{N}\cup \{0\}$. In odd dimensions, Theorem~\ref{t:MT} follows directly from the identity \eqref{eq:Identity0}. Indeed, for $k=n-2$, the first term on the right-hand side vanishes, and the remaining term is given by
$(n-2)f^{\,n-3}h^{\,l+1}$. Integrating \eqref{eq:Identity0} over $\mathbb{S}^{n-1}$ with $l=1$ yields $fh=0$. Applying \eqref{eq:Identity0} once more with $k=l=1$ yields $h=0$. Finally, integrating \eqref{eq:Identity0} for $k\neq n-2$ and $l=0$ implies $f=0$. 

In even dimensions, we first show that $f\le 0$ on $\mathbb S^{n-1}$. The non-positivity of $f$, together with the second and third equations of \eqref{eq:HE0}, implies $f=0$ and $h=0$. We argue by contradiction, supposing that $\Omega=\{f>0\}\neq \emptyset$. By integrating \eqref{eq:Identity0} over $\Omega$ and arguing as in the odd-dimensional case, we deduce a contradiction $h=f=0$ on $\Omega=\{f>0\}$.

\subsubsection{Key ideas of the proof of Theorem \ref{t:MT2}}

The proofs of Theorems~\ref{t:MT2} and \ref{t:MT3}
are based on analyzing system \eqref{eq:HE0}
under the assumption that $v$ is geodesible. Namely, 
\begin{equation} 
\begin{aligned}
(1-\alpha) f\boldsymbol{v}+\frac{1}{2}d h&=0,\\
\boldsymbol{v}\cdot \nabla_{S}f-(1-\alpha)|\boldsymbol{v}|^{2}-\alpha |f|^{2}-\alpha h&=0,\\
(n-1-\alpha)f&=0.
\end{aligned}
\label{eq:HE0b}
\end{equation}
For $\alpha\neq n-1$ and $1$, $f=0$, $dh=0$ and $|\boldsymbol{v}|$ is constant. If $n$ is odd, then $\mathbb{S}^{n-1}$ is even-dimensional, and the Poincar\'{e}--Hopf theorem implies that a nonvanishing vector field of constant magnitude cannot exist. For even dimensions, $\alpha\neq 0$ for nonzero $|\boldsymbol{v}|$ and $q=-\frac{1}{2\alpha}|\boldsymbol{v}|^{2}$.

\subsubsection{Key ideas of the proof of Theorem \ref{t:MT3}}

If $\boldsymbol{u}$ is geodesible on $\mathbb{R}^{n}\setminus \{0\}$, $\alpha \Pi=0$. Namely, $\alpha (h+f^{2})=0$. Integrating the second equation of \eqref{eq:HE0b} and using $\textrm{div}_S\ v=0$, we conclude that $\alpha=1$.

\subsection{Organization of the paper}

This paper is organized as follows. In \S \ref{S2}, we derive system~\eqref{eq:HE0} for $(-\alpha)$-homogeneous solutions and establish a rigidity theorem for irrotational solutions satisfying $du=0$. In \S \ref{S3}, we prove the rigidity theorem for $(-1)$-homogeneous solutions (Theorem~\ref{t:MT}), and also establish rigidity in the cases $n=2$ and $n=3$. In \S \ref{S4}, assuming that the tangential vector field $\boldsymbol{v}$ is geodesible, we use~\eqref{eq:HE0b} to characterize the values of $\alpha$ for which $\boldsymbol{u}$ is a $(-\alpha)$-homogeneous Euler solution or a $(-\alpha)$-homogeneous geodesible solution, thereby proving Theorems~\ref{t:MT2} and~\ref{t:MT3}. In \S \ref{S5}, we prove the nonexistence of geodesible $(-\alpha)$-homogeneous solutions for $1<\alpha\le n-1$. Finally, in \S \ref{S6}, we discuss several open questions concerning the topology and rigidity of higher-dimensional Euler flows.

\subsection{Acknowledgements}
The first author's research on the Navier–Stokes equations has been greatly influenced by the book by H. Sohr \cite{Sohr}, to which he is deeply indebted. This work was initiated during his visit to Shanghai Jiao Tong University in August 2025, whose warm hospitality is gratefully acknowledged. KA was supported by the JSPS through the Grant in Aid for Scientific Research (C) 24K06800, MEXT Promotion of Distinctive Joint Research Center Program JPMXP0723833165, and Osaka Metropolitan University Strategic Research Promotion Project (Development of International Research Hubs). NS was partially supported by JSPS KAKENHI Grant 
No. 25K07267, No. 22H04936, and No. 24K00615. The research of Xie was  partially supported by  NSFC grants 12571238 and 12426203. Part of the work was done when Xie was visiting Osaka Metropolitan University in November 2025 and he would like to thank the support and hospitality of Osaka Metropolitan University. 

\section{$(-\alpha)$-homogeneous solutions}\label{S2}

We derive the equations on the sphere $\mathbb{S}^{n-1}$ satisfied by $(-\alpha)$-homogeneous solutions to the Euler equations on $\mathbb{R}^n \setminus \{0\}$, and show that, in the irrotational case, such $(-\alpha)$-homogeneous solutions exist only for discrete values of $\alpha$. In the case of the Navier--Stokes equations, when $\alpha=1$, an analogous system can be obtained by computing the tangential and normal components of the Laplacian term.

\subsection{The equations on the sphere}
We use polar coordinates \(\rho = |\boldsymbol{x}| > 0\) and \(\boldsymbol{y} = \frac{\boldsymbol{x}}{|\boldsymbol{x}|} = \boldsymbol{e}_{\rho} \in \mathbb{S}^{n-1}\) for \(\boldsymbol{x} \in \mathbb{R}^n\). A \((-\alpha)\)-homogeneous vector field \(\boldsymbol{u}\) and a \(( -2\alpha)\)-homogeneous scalar function \(p\) on \(\mathbb{R}^n \setminus \{0\}\) are written as
\begin{align}
\boldsymbol{u}(\boldsymbol{x})
=
\frac{1}{\rho^{\alpha}}
\bigl(\boldsymbol{v}(\boldsymbol{y}) + f(\boldsymbol{y}) \boldsymbol{y}\bigr),
\qquad
p
=
\frac{1}{\rho^{2\alpha}} q(\boldsymbol{y}),
\label{eq:HV}
\end{align}
where \(\boldsymbol{v}\) is a tangential vector field on \(\mathbb{S}^{n-1}\). Let \(v = \boldsymbol{v}^{\flat}_{\mathbb{S}^{n-1}}\) denote the dual one-form of \(\boldsymbol{v}\) with respect to the standard metric on \(\mathbb{S}^{n-1}\). Then, using the relation
\[
\boldsymbol{v}^{\flat}_{\mathbb{R}^n} = \rho\, v,
\]
the one-form \(u = \boldsymbol{u}^{\flat}_{\mathbb{R}^n}\) can be written as
\begin{align}
u
=
\frac{1}{\rho^{\alpha}}
\bigl(\rho\, v(\boldsymbol{y}) + f(\boldsymbol{y})\, d\rho\bigr)
=
\frac{1}{\rho^{\alpha-1}} v(\boldsymbol{y}) + \frac{1}{\rho^{\alpha}} f(\boldsymbol{y})\, d\rho.  \label{eq:HU}
\end{align}
Here, we identify \(\boldsymbol{y}\) with \(\partial_{\rho}\). Moreover, we introduce the rescaled vector field \(\tilde{\boldsymbol{v}} = \rho \boldsymbol{v}\), so that
\begin{align}
\boldsymbol{v} = \frac{1}{\rho} \tilde{\boldsymbol{v}},
\qquad
\tilde{\boldsymbol{v}}^{\flat}_{\mathbb{S}^{n-1}} = v.
\label{eq:N}
\end{align}

We express the exterior derivative and the gradient in \(\mathbb{R}^n\) as
\begin{align*}
d = d\rho\, \partial_{\rho} + d_{S},
\qquad
\nabla = \boldsymbol{y}\, \partial_{\rho} + \frac{1}{\rho}\, \nabla_{S},
\end{align*}
where \(d_S\) and \(\nabla_S\) denote the exterior derivative and the gradient on \(\mathbb{S}^{n-1}\), respectively.

\begin{ex}
In polar coordinates \((\rho,\theta,\phi)\), we write
\begin{align*}
d_{S} = d\theta\, \partial_{\theta} + d\phi\, \partial_{\phi}, 
\qquad
\nabla_{S} = \boldsymbol{e}_{\theta}\, \partial_{\theta} + \boldsymbol{e}_{\phi}\, \frac{1}{\sin\theta}\, \partial_{\phi},
\end{align*}
where
\[
\boldsymbol{e}_{\theta} ={}^{t} (\cos\phi \cos\theta, \sin\phi \cos\theta, -\sin\theta), 
\qquad
\boldsymbol{e}_{\phi} ={}^{t} (-\sin\phi, \cos\phi, 0).
\]
The orthonormal basis \(\{\boldsymbol{e}_{\rho}, \boldsymbol{e}_{\theta}, \boldsymbol{e}_{\phi}\}\) is identified with
\begin{align*}
\boldsymbol{e}_{\rho} = \partial_{\rho}, 
\qquad
\boldsymbol{e}_{\theta} = \frac{1}{\rho}\, \partial_{\theta}, 
\qquad
\boldsymbol{e}_{\phi} = \frac{1}{\rho \sin\theta}\, \partial_{\phi}.
\end{align*}
Their dual one-forms in \(\mathbb{R}^3\) are given by
\begin{align*}
\boldsymbol{e}_{\rho}^{\flat} = d\rho, 
\qquad
\boldsymbol{e}_{\theta}^{\flat} = \rho\, d\theta, 
\qquad
\boldsymbol{e}_{\phi}^{\flat} = \rho \sin\theta\, d\phi.
\end{align*}
Accordingly, a tangential vector field \(\boldsymbol{v}\) on \(\mathbb{S}^2\) can be written as
\begin{align*}
\boldsymbol{v}
=
v^{\theta}(\theta,\phi)\, \boldsymbol{e}_{\theta}
+
v^{\phi}(\theta,\phi)\, \boldsymbol{e}_{\phi}
=
v^{\theta}\frac{1}{\rho}\partial_{\theta}
+
v^{\phi}\frac{1}{\rho\sin\theta}\partial_{\phi}
=
\frac{1}{\rho}\tilde{\boldsymbol{v}},
\end{align*}
where
\[
\tilde{\boldsymbol{v}}
=
v^{\theta}\partial_{\theta}
+
v^{\phi}\frac{1}{\sin\theta}\partial_{\phi}.
\]
Hence, the associated one-form on \(\mathbb{S}^2\) is given by
\begin{align*}
v
=
\tilde{\boldsymbol{v}}^{\flat}_{\mathbb{S}^2}
=
v^{\theta} d\theta + v^{\phi} \sin\theta\, d\phi.
\end{align*}
\end{ex}

 We first compute the divergence of \(\boldsymbol{u}\):
\begin{align*}
\operatorname{div}\boldsymbol{u}
=
\frac{1}{\rho^{\alpha+1}}
\Bigl(\operatorname{div}_{S} \boldsymbol{v} +(n-1-\alpha)f\Bigr).
\end{align*}
We next compute the vorticity two-form. From \eqref{eq:HU}, we obtain
\begin{align}
d u
=
\frac{1}{\rho^{\alpha-1}} d_{S}v
+
\frac{1}{\rho^{\alpha}} d\rho \wedge A,
\qquad
A=(1-\alpha) v- d_{S} f.
\label{eq:V}
\end{align}
We now compute \(\iota_{\boldsymbol{u}}du\). Using \eqref{eq:HV}, we decompose
\begin{align*}
\iota_{\boldsymbol{u}}du
=
\frac{1}{\rho^{\alpha}}
\Bigl(
\iota_{\boldsymbol{v}}du
+
f\,\iota_{\partial_{\rho}}du
\Bigr).
\end{align*}
Since \(d_S v\) is purely tangential, we have \(\iota_{\partial_\rho} d_S v = 0\), and hence
\begin{align*}
\iota_{\boldsymbol{v}}du
=
\frac{1}{\rho^{\alpha-1}} \iota_{\boldsymbol{v}} d_S v
+
\frac{1}{\rho^{\alpha}} \iota_{\boldsymbol{v}}(d\rho \wedge A) 
=
\frac{1}{\rho^{\alpha-1}} \iota_{\boldsymbol{v}} d_S v
-
\frac{1}{\rho^{\alpha}} (\iota_{\boldsymbol{v}} A)\, d\rho,
\end{align*}
where \(\iota_{\boldsymbol{v}} d\rho = 0\) has been used. Similarly, one has
\begin{align*}
\iota_{\partial_\rho}du
=
\frac{1}{\rho^{\alpha}} A.
\end{align*}
Combining these yields
\begin{align*}
\iota_{\boldsymbol{u}}du
&=
\frac{1}{\rho^{2\alpha-1}} \iota_{\boldsymbol{v}} d_S v
-
\frac{1}{\rho^{2\alpha}} (\iota_{\boldsymbol{v}} A)\, d\rho
+
\frac{1}{\rho^{2\alpha}} f A.
\end{align*}
Using \(\tilde{\boldsymbol{v}} = \rho \boldsymbol{v}\) in \eqref{eq:N}, we rewrite \(\iota_{\boldsymbol{u}}du\) as
\begin{align*}
\iota_{\boldsymbol{u}}d u
&=
\frac{1}{\rho^{2\alpha}}
\Bigl(
\iota_{\tilde{\boldsymbol{v}}} d_{S} v
+
f A
\Bigr)
-
\frac{1}{\rho^{2\alpha+1}}
\bigl(\iota_{\tilde{\boldsymbol{v}}} A\bigr)\, d\rho,
\end{align*}
For \(\Pi(\boldsymbol{x})=\rho^{-2\alpha}\Pi(\boldsymbol{y})\), we have
\begin{align*}
d\Pi
=
-2\alpha \rho^{-2\alpha-1}\Pi\, d\rho
+
\rho^{-2\alpha} d_S\Pi.
\end{align*}
Taking the radial and spherical components of \eqref{eq: ER}, 
\begin{align*}
\iota_{\tilde{\boldsymbol{v}}}d_S v+fA+d_S\Pi&=0,\\
\iota_{\tilde{\boldsymbol{v}}}A+2\alpha \Pi
&=0.
\end{align*}
We use the identity for the covariant derivative along with $\tilde{\boldsymbol{v}}$,
\begin{align}
(\nabla_{\tilde{\boldsymbol{v}}}\tilde{\boldsymbol{v}})^{\flat}=\iota_{\tilde{\boldsymbol{v}}}d_S v+d_S \frac{1}{2}|\tilde{\boldsymbol{v}}|^{2}.\label{eq:covariant}
\end{align}
Using the definition of \(\Pi\) and identifying \(\boldsymbol{v} = \tilde{\boldsymbol{v}}\) on \(\mathbb{S}^{n-1}\), we drop the tilde and obtain the following system governing the profile \((\boldsymbol{v}, f, q)\) on \(\mathbb{S}^{n-1}\):
\begin{equation} 
\begin{aligned}
(1-\alpha) f\boldsymbol{v}+ \nabla_{\boldsymbol{v}} \boldsymbol{v}+\nabla_{S} q&=0,\\
\nabla_{\boldsymbol{v}}f-|\boldsymbol{v}|^{2}-\alpha |f|^{2}-2\alpha q&=0,\\
\textrm{div}_S\ \boldsymbol{v}+(n-1-\alpha)f&=0.
\end{aligned}
\label{eq: HE}
\end{equation}

The covariant derivative of a scalar function \(f\) and a vector field \(\boldsymbol{v}\) on \(\mathbb{S}^{n-1}\) are given by
\[
\nabla_{\boldsymbol{v}} f = \boldsymbol{v} \cdot \nabla_{S} f,
\qquad
\nabla_{\boldsymbol{v}} \boldsymbol{v}
=
\boldsymbol{v} \cdot \nabla_{S} \boldsymbol{v}
+
|\boldsymbol{v}|^{2}\boldsymbol{y}.
\]
Here, the covariant divergence of \(\boldsymbol{v}\) is given by
\[
\operatorname{div}_{S} \boldsymbol{v}
=
\nabla_{S} \cdot \boldsymbol{v}.
\]
Henceforth, we write \(d\) in place of \(d_S\).

\begin{rem}[The Bernoulli function]
Taking the inner product of the first equation in \eqref{eq: ER} with \(\boldsymbol{u}\), we obtain
\[
\boldsymbol{u} \cdot \nabla \Pi = 0.
\]
For a \((-\alpha)\)-homogeneous vector field \(\boldsymbol{u}\), the function \(\Pi\) is \((-2\alpha)\)-homogeneous and satisfies
\begin{equation}
\nabla_{\boldsymbol{v}} \Pi = 2\alpha f \Pi.
\label{eq:1stI}
\end{equation}
\end{rem}

\begin{rem}[The vorticity function]
For even dimensions \(n=2m\), the vorticity function \(W = (du)^{m}/\mu\) is a first integral of \(\boldsymbol{u}\) (e.g., \cite[II.6.8]{AK21}), that is,
\[
\boldsymbol{u} \cdot \nabla W = 0.
\]
For a \((-\alpha)\)-homogeneous vector field \(\boldsymbol{u}\), the function \(W\) is \((-(\alpha+3)m+1)\)-homogeneous and satisfies
\begin{align}
\nabla_{\boldsymbol{v}} W = ((\alpha+3)m-1)\, f W.
\label{eq: W}
\end{align}
\end{rem}

\begin{rem}[The Navier--Stokes case]
The equations for the profile \((\boldsymbol{v}, f, q)\) on \(\mathbb{S}^{n-1}\) corresponding to \(( -1)\)-homogeneous solutions (\(\alpha = 1\)) of the Navier--Stokes equations in \(\mathbb{R}^{n} \setminus \{0\}\) are obtained by adding the tangential and normal components of \(\Delta \boldsymbol{u}\) to the right-hand sides of \eqref{eq: ER}, respectively \cite{Sverak2011}. Namely, it is of the form 
\begin{equation} 
\begin{aligned}
\nabla_{\boldsymbol{v}} \boldsymbol{v}+\nabla_{S} q&=\Delta_{H}\boldsymbol{v}+2\nabla_{S}f ,\\
\nabla_{\boldsymbol{v}}f-|\boldsymbol{v}|^{2}- |f|^{2}-2 q&=\Delta_{S}f,\\
\textrm{div}_S\ \boldsymbol{v}+(n-2)f&=0.
\end{aligned}
\label{eq: HNS}
\end{equation}
Here, $\Delta_{S}=\nabla_{S}\cdot \nabla_{S}$ is the Laplace-Beltrami operator on $\mathbb{S}^{n-1}$ and $-\Delta_{H}\boldsymbol{v}=((d\delta +\delta d)\boldsymbol{v}^{\flat})^{\sharp}$ is the Hodge Laplacian for 1-form/vector fields on $\mathbb{S}^{n-1}$ with the adjoint operator $\delta=d^{*}$. We provide a proof for \eqref{eq: HNS} in Appendix \ref{App}.
\end{rem}

\subsection{The irrotational case}

We begin with irrotational solutions $du=0$ to \eqref{eq: HE} with $d \Pi=0$. By the representation \eqref{eq:V} and the divergence-free condition $\nabla \cdot \boldsymbol{u}=0$, the profile $(\boldsymbol{v},f,q)$ of irrotational solutions satisfy
\begin{equation}
\begin{aligned}
d v&=0,\\
(\alpha-1)v+df&=0, \\
\textrm{div}_S\ \boldsymbol{v}+(n-1-\alpha)f&=0.
\end{aligned}
\label{eq: IR}
\end{equation}
By homogeneity, the condition $d\Pi=0$ can be written as 
\begin{align}
\alpha(2q+|\boldsymbol{v}|^{2}+|f|^{2})=0.   \label{eq:Pi}
\end{align}
We use the eigenvalues and eigenfunctions (spherical harmonics) of the Laplace–Beltrami operator on $\mathbb{S}^{n-1}$ \cite[Proposition 3.5]{AH12}.

\begin{prop}
For each $l \in \mathbb{N}_0$, there exist spherical harmonics satisfying
\[
-\Delta_{S} Y = l(l+n-2) Y.
\]
The multiplicity of this eigenvalue is given by
\begin{align}
\frac{(2l+n-2)(l+n-3)!}{l!(n-2)!}. \label{eq:M}
\end{align}
\end{prop}

\begin{thm}\label{t:irrotational}
Irrotational $(-\alpha)$-homogeneous solutions $(\boldsymbol{u},p)\in C^{1}(\mathbb{R}^{n}\setminus\{0\})$ to \eqref{eq: ER} exist if and only if one of the following cases holds:

\medskip
\noindent
(i) $\alpha=1$ and $n=2$. In this case, the irrotational solution is given by
\begin{align*}
\boldsymbol{u}=\frac{1}{\rho}(\boldsymbol{v}+f\boldsymbol{y}), 
\qquad 
p=\frac{1}{\rho^{2}}q,
\end{align*}
where $\boldsymbol{v}=v^{\theta}\boldsymbol{e}_{\theta}$ for $\boldsymbol{e}_{\theta}={}^{t}(-\sin\theta,\cos\theta)$ and $v^{\theta}, f, q \in \mathbb{R}$ are constants.

\medskip
\noindent
(ii) $\alpha \neq 1$ and $n\geq 2$, with
\begin{align}
\alpha = n-1+l,\quad 1-l,  \qquad l \in \mathbb{N}_0. \label{eq:alphaIR}
\end{align}
In this case, the irrotational solution is given by
\begin{align*}
\boldsymbol{u}=\frac{1}{\rho^{\alpha}}(\boldsymbol{v}+f\boldsymbol{y}), 
\qquad 
p=\frac{1}{\rho^{2\alpha}}q,
\end{align*}
where
\begin{align}
\boldsymbol{v}=\frac{1}{1-\alpha}\nabla_{S} f, \label{eq:v}
\end{align}
and $f$ is a spherical harmonic satisfying
\[
-\Delta_{S} f = (1-\alpha)(n-1-\alpha) f.
\]
The pressure is given by
\begin{equation}
q = -\frac{1}{2}\bigl(|\boldsymbol{v}|^{2}+f^{2}\bigr).  \label{eq:pressure}
\end{equation}
In the case $\alpha=0$, $q$ is determined up to an additive constant. The number of linearly independent such functions $f$ is given by \eqref{eq:M}.
\end{thm}

\begin{proof}
(i) For $\alpha = 1$ and $n = 2$, $f$ is constant by the second equation in \eqref{eq: IR}. 
For $\boldsymbol{v} = v^{\theta} e_{\theta}$, the divergence-free condition
\[
0 = \operatorname{div}_{S} \boldsymbol{v} = \partial_{\theta} v^{\theta}
\]
implies that $v^{\theta}$ is constant. 
Finally, by \eqref{eq:Pi}, $q$ is constant.

(ii) For $\alpha = 1$, $f$ is constant by \eqref{eq: IR}. 
Using $n \geq 3$ and integrating the third equation in \eqref{eq: IR} over $\mathbb{S}^{n-1}$, we obtain $f = 0$. 
Thus $d v = 0$ and $\operatorname{div}_{S} \boldsymbol{v} = 0$. 
Since the first cohomology group of $\mathbb{S}^{n-1}$ is trivial, it follows that $\boldsymbol{v} = 0$.

For $\alpha \neq 1$, $\boldsymbol{v}$ is given by \eqref{eq:v} from the second equation in \eqref{eq: IR}. 
Taking the divergence of \eqref{eq:v} and using the third equation in \eqref{eq: IR}, we find that $f$ is a spherical harmonic. 
Thus, $(1-\alpha)(n-1-\alpha) = l(l+n-2)$ for some $l \in \mathbb{N}_0$, and the number of linearly independent such functions $f$ is given by \eqref{eq:M}. 
Solving this quadratic equation for $\alpha$ yields \eqref{eq:alphaIR}. 
Moreover, by \eqref{eq:Pi}, $q$ is given by \eqref{eq:pressure}. 
In the case $\alpha = 0$, the quantity
\[
2\Pi = 2q + |\boldsymbol{v}|^{2} + f^{2}
\]
is constant, and hence $q$ is determined by \eqref{eq:pressure} up to an additive constant.
\end{proof}

\section{$(-1)$-homogeneous solutions and proof of Theorem \ref{t:MT}}\label{S3}

We prove the rigidity of $(-1)$-homogeneous solutions in dimensions $n \geq 4$ (Theorem \ref{t:MT}). 
We also establish corresponding rigidity results in the cases $n = 2$ and $n = 3$. 
When $\alpha = 1$, the system \eqref{eq: HE} reduces to
\begin{equation}
\begin{aligned}
\nabla_{\boldsymbol{v}} \boldsymbol{v} + \nabla_{S} q &= 0,\\
\nabla_{\boldsymbol{v}} f - |\boldsymbol{v}|^{2} - f^{2} - 2q &= 0,\\
\operatorname{div}_{S} \boldsymbol{v} + (n-2)f &= 0.
\end{aligned}
\label{eq:-1}
\end{equation}
The first integral condition \eqref{eq:1stI} takes the form
\begin{align}
\nabla_{\boldsymbol{v}} \Pi = 2 f \Pi. \label{eq:1st}
\end{align}
Using the identity for the covariant derivative \eqref{eq:covariant} and setting $h = 2q + |\boldsymbol{v}|^{2}$, 
the first and second equations in \eqref{eq:-1} can be rewritten as
\begin{equation}
\begin{aligned}
\iota_{\boldsymbol{v}} d v + \frac{1}{2} d h &= 0,\\
\nabla_{\boldsymbol{v}} f - f^{2} - h &= 0.
\end{aligned}
\label{eq:dh}
\end{equation}
Taking the inner product of the first equation with $\boldsymbol{v}$, we obtain
\begin{align}
\nabla_{\boldsymbol{v}} h = 0. \label{eq:1stv}
\end{align}

The system \eqref{eq:-1} with $f = 0$ and $h = 0$ reduces to the Euler equations on $\mathbb{S}^{n-1}$ with vanishing Bernoulli function. 
Namely, $\boldsymbol{v}$ is geodesible on $\mathbb{S}^{n-1}$:
\begin{equation}
\iota_{\boldsymbol{v}} d v = 0, \quad
\operatorname{div}_{S} \boldsymbol{v} = 0.
\label{eq:B}
\end{equation}
Thus, the set of geodesible vector fields on $\mathbb{S}^{n-1}$ is contained in the class of $(-1)$-homogeneous Euler flows in $\mathbb{R}^{n} \setminus \{0\}$. 
We show that the converse holds.

\begin{thm}\label{t:RH}
Let $(\boldsymbol{u},p) \in C^{1}(\mathbb{R}^{n}\setminus\{0\})$ be a $(-1)$-homogeneous solution of \eqref{eq: ER}.

\medskip
\noindent
(i) If $n \geq 4$, then $(\boldsymbol{u},p)$ satisfies \eqref{eq:HV} with $f = 0$ and $q = -|\boldsymbol{v}|^{2}/2$, where $\boldsymbol{v}$ is a geodesible vector field on $\mathbb{S}^{n-1}$ satisfying \eqref{eq:B}.

\medskip
\noindent
(ii) If $n = 3$, then $(\boldsymbol{u},p) = (0,0)$.

\medskip
\noindent
(iii) If $n = 2$, then $(\boldsymbol{u},p)$ is irrotational and satisfies \eqref{eq:HV}, where $\boldsymbol{v} = v^{\theta}\boldsymbol{e}_{\theta}$ and $v^{\theta}, f, q \in \mathbb{R}$ are constants.
\end{thm}

\begin{proof}
We first consider the case $n \geq 3$. We work with the system
\begin{equation}
\begin{aligned}
\boldsymbol{v}\cdot \nabla_{S} h &= 0,\\
\boldsymbol{v}\cdot \nabla_{S} f &= f^{2} + h,\\
\operatorname{div}_{S} \boldsymbol{v} &= (2-n)f.
\end{aligned}
\label{eq:3eq}
\end{equation}
We observe that for integers $k \in \mathbb{N}$ and $l \in \mathbb{N}_{0}$,
\begin{align}
\operatorname{div}_{S} (\boldsymbol{v} f^{k} h^{l}) 
= (2-n+k) f^{k+1} h^{l} + k f^{k-1} h^{l+1}.
\label{eq:identity}
\end{align}

We first show that $f \leq 0$ on $\mathbb{S}^{n-1}$. 
Suppose, for contradiction, that there exists a point $y_{0} \in \mathbb{S}^{n-1}$ such that $f(y_{0}) > 0$. 
By continuity, $f$ is positive in a neighborhood of $y_{0}$. 
Let $\Omega$ be the largest connected open subset of $\mathbb{S}^{n-1}$ containing $y_{0}$ on which $f > 0$. 
By the maximality of $\Omega$, we have $f = 0$ on $\partial \Omega$. 
By possibly replacing $\Omega$ with $\Omega_{\varepsilon} = \{ f > \varepsilon \}$ for some $\varepsilon > 0$, we may assume that $\partial \Omega$ is of class $C^{1}$ and that $\nabla_{S} f \neq 0$ on $\partial \Omega$.

Choosing $k = n-2$ and $l = 1$ in \eqref{eq:identity}, the first term on the right-hand side vanishes. 
Integrating over $\Omega$ and applying the divergence theorem, we obtain
\begin{align*}
\int_{\Omega} f^{n-3} h^{2} \, \mu_{S} = 0,
\end{align*}
where $\mu_{S}$ denotes the volume form on $\mathbb{S}^{n-1}$. 
Since $f > 0$ on $\Omega$, it follows that $h = 0$ on $\Omega$.

Next, choosing $k \neq n-2$ and $l = 0$ in \eqref{eq:identity}, we obtain
\begin{align*}
\int_{\Omega} f^{k+1} \, \mu_S = 0,
\end{align*}
which contradicts the fact that $f > 0$ on $\Omega$. 
Hence, $f \leq 0$ on $\mathbb{S}^{n-1}$.

Integrating the third equation of \eqref{eq:3eq} over $\mathbb{S}^{n-1}$, we obtain
\begin{align*}
0=\int_{\mathbb{S}^{n-1}} \nabla_{S}\cdot \boldsymbol{v}
= (2-n)\int_{\mathbb{S}^{n-1}} f,
\end{align*}
and thus $\int_{\mathbb{S}^{n-1}} f=0$. Since $f\leq 0$, we conclude that $f=0$ on $\mathbb{S}^{n-1}$. Then the second equation of \eqref{eq:3eq} implies $h=0$. Hence $\boldsymbol{v}$ is geodesible and satisfies \eqref{eq:B}. For $n=3$, the condition $\iota_{\boldsymbol{v}}dv=0$ implies $dv=0$, and hence $v=0$, since any closed one-form on $\mathbb{S}^{2}$ is exact and thus must vanish. This proves (i) and (ii).

It remains to consider the case $n=2$. From the third equation of \eqref{eq:3eq}, we have
\begin{align*}
0=\nabla_{S}\cdot \boldsymbol{v}=\partial_{\theta}v^{\theta},
\end{align*}
and thus $v^{\theta}$ is constant. Since $dv=0$, the first equation of \eqref{eq:dh} implies that $h$ is constant, and hence $q$ is also constant. If $v^{\theta}=0$, then the second equation of \eqref{eq:3eq} implies that $f$ is constant. 

Assume $v^{\theta}\neq 0$. Then
\begin{align*}
\partial_{\theta}f=\frac{1}{v^{\theta}}(f^{2}+h).
\end{align*}
Let $M=\sup_{\mathbb{S}^{1}} f$. At a maximum point, we have $M^{2}+h=0$, and hence $f^{2}+h=f^{2}-M^{2}\leq 0$ on $\mathbb{S}^{1}$. Thus $\partial_{\theta}f$ does not change sign. Since $f$ is periodic, it follows that $f$ is constant. This proves (iii).
\end{proof}

\begin{proof}[Proof of Theorem \ref{t:MT}]
The result follows from Theorem \ref{t:RH}.
\end{proof}

\section{Extensions via geodesible fields}\label{S4}

\subsection{$(-\alpha)$-homogeneous Euler}

We prove the rigidity of $(-\alpha)$-homogeneous solutions to \eqref{eq: ER} in Theorems \ref{t:MT2} and \ref{t:MT3}. 
Setting $h = 2q + |\boldsymbol{v}|^{2}$, the system \eqref{eq: HE} can be rewritten as
\begin{equation}
\begin{aligned}
(1-\alpha) f \boldsymbol{v} + \iota_{\boldsymbol{v}} d\boldsymbol{v} + \frac{1}{2} d h &= 0,\\
\nabla_{\boldsymbol{v}} f - (1-\alpha)|\boldsymbol{v}|^{2} - \alpha f^{2} - \alpha h &= 0,\\
\operatorname{div}_{S} \boldsymbol{v} + (n-1-\alpha) f &= 0.
\end{aligned}
\label{eq:HE2}
\end{equation}
If $\boldsymbol{v}$ is geodesible, i.e. satisfies \eqref{eq:B}, then the system reduces to
\begin{equation}
\begin{aligned}
(1-\alpha) f \boldsymbol{v} + \frac{1}{2} d h &= 0,\\
\nabla_{\boldsymbol{v}} f - (1-\alpha)|\boldsymbol{v}|^{2} - \alpha f^{2} - \alpha h &= 0,\\
(n-1-\alpha) f &= 0.
\end{aligned}
\label{eq:HE2g}
\end{equation}

\begin{proof}[Proof of Theorem \ref{t:MT2}]
Suppose that $(\boldsymbol{u},p)$ defined by \eqref{eq:GF2} is a $(-\alpha)$-homogeneous solution of \eqref{eq: ER}. 
Then $(\boldsymbol{v},f,q)$ satisfies the system \eqref{eq:HE2g}. 
Since $\alpha \neq n-1$, we obtain $f = 0$. 
Substituting this into the first and second equations yields $d h = 0$ and
\[
(1-\alpha)|\boldsymbol{v}|^{2} + \alpha h = 0.
\]
Since $\alpha \neq 1$, it follows that $|\boldsymbol{v}|$ is constant. 
Since $\boldsymbol{v} \neq 0$, it follows that $\alpha \neq 0$, and hence
\[
q = -\frac{1}{2\alpha}|\boldsymbol{v}|^{2}.
\]
The Bernoulli function $\Pi$ is given by \eqref{eq:BH}. On even-dimensional spheres, any smooth nonvanishing vector field does not exist by the Poincar\'{e}--Hopf theorem; in particular, a vector field with constant nonzero magnitude cannot exist.
\end{proof}

\subsection{$(-\alpha)$-homogeneous geodesible}

It remains to show Theorem \ref{t:MT3}. For a $(-\alpha)$-homogeneous geodesible field $\boldsymbol{u}$ in $\mathbb{R}^{n}\setminus \{0\}$, the Bernoulli function $\Pi$ satisfies $\alpha\Pi=0$ since $\Pi$ vanishes for $\alpha\neq 0$ and is constant for $\alpha=0$. This condition is expressed as 
\begin{align}
\alpha(h+f^{2})=0.  \label{eq:geodesible}
\end{align}
Under this condition, the system \eqref{eq: HE} reduces to
\begin{equation}
\begin{aligned}
(1-\alpha) f\boldsymbol{v}+ \iota_{\boldsymbol{v}}d v+\frac{1}{2}dh&=0,\\
\nabla_{\boldsymbol{v}}f-(1-\alpha)|\boldsymbol{v}|^{2}&=0,\\
\operatorname{div}_S \boldsymbol{v}+(n-1-\alpha)f&=0.
\end{aligned}
\label{eq: B=0}
\end{equation}

\begin{proof}[Proof of Theorem \ref{t:MT3}]
Suppose that $\boldsymbol{u}$ is a $(-\alpha)$-homogeneous geodesible field given by \eqref{eq:GF2} for $\alpha\neq 1$ with a geodesible vector field $\boldsymbol{v}$ on $\mathbb{S}^{n-1}$. By the second equation in \eqref{eq: B=0},
\begin{align*}
\boldsymbol{v}\cdot \nabla_{S}f=(1-\alpha)|\boldsymbol{v}|^{2}.
\end{align*}
Integrating this identity over $\mathbb{S}^{n-1}$ and using $\textrm{div}_{S}\ \boldsymbol{v}=0$, we find that $\boldsymbol{v}=0$. 
\end{proof}

\section{$(-\alpha)$-homogeneous geodesible solutions}\label{S5}

We remark that the rigidity of $(-1)$-homogeneous solutions to \eqref{eq: HE} in Theorem \ref{t:MT} extends to the case $1 < \alpha \leq n-1$ for $(-\alpha)$-homogeneous geodesible solutions.

\begin{thm}\label{t:5.1}
Let $n \geq 3$. 
For $1 < \alpha \leq n-1$, there are no $(-\alpha)$-homogeneous geodesible solutions $(\boldsymbol{u},p) \in C^{1}(\mathbb{R}^{n}\setminus\{0\})$ to \eqref{eq: HE}, except for the irrotational solution in the case $\alpha = n-1$.
\end{thm}

\begin{proof}
By the second and the third equations in \eqref{eq: B=0},
\begin{equation} 
\nabla_{S}\cdot (\boldsymbol{v}f)=(1-\alpha)|\boldsymbol{v}|^{2}
+(\alpha+1-n)f^{2}.
\end{equation}
For $1<\alpha< n-1$, the coefficients of the right-hand sides are negative. By integrating this identity on $\mathbb{S}^{n-1}$, $\boldsymbol{v}=0$ and $f=0$ follow. For $\alpha=n-1$, $\boldsymbol{v}=0$ and $f$ is constant. Such $\boldsymbol{u}$ is irrotational and given in Theorem \ref{t:irrotational} (ii).
\end{proof}

\begin{rem}\label{rem}
For $n = 3$, there are no $(-\alpha)$-homogeneous rotational Beltrami solutions $(\boldsymbol{u},p) \in C^{2}(\mathbb{R}^{3}\setminus\{0\})$ for $\alpha < 1$ \cite[Proposition 3.2]{Shv} (see also \cite[Theorem 2.5]{Abe11}). 
The proof relies on an identity involving the radial component of the vorticity $\nabla \times \boldsymbol{u}$. 
In dimensions $n \geq 4$, it is unclear whether a similar identity holds for the two-form $du$.

On the other hand, for $\alpha>2$, axisymmetric $(-\alpha)$-homogeneous Beltrami fields $(\boldsymbol{u},p)\in C^{1}(\mathbb{R}^{3}\setminus\{0\})$ do exist \cite[Theorem 1.1]{Abe11}. Continuous axisymmetric $(-\alpha)$-homogeneous Beltrami fields $(u,p)\in C(\mathbb{R}^{3})$ also exist for $\alpha<0$. Their construction is based on the Grad–Shafranov equation associated with the axisymmetric stream function; see Figure \ref{f2}.
\end{rem}

\begin{figure}[h]
\centering
\includegraphics[width=0.85\linewidth]{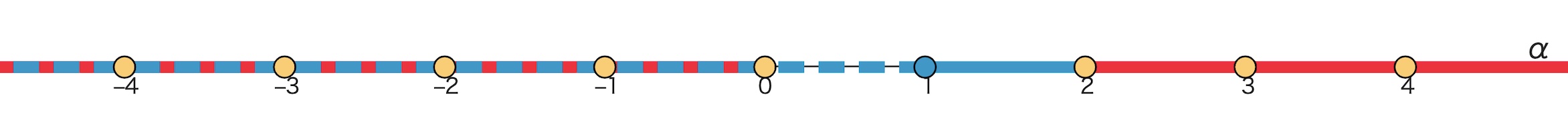}
\caption{The existence and nonexistence ranges of $\alpha\in \mathbb{R}$ for $(-\alpha)$-homogeneous Beltrami solutions in $\mathbb{R}^{3}\setminus\{0\}$. The blue line indicates the range of $\alpha$ for which no solutions exist, while the red line indicates the range for which solutions exist. The yellow dots represent the integer values of $\alpha$ corresponding to irrotational solutions.}\label{f2}
\end{figure}

In Theorem \ref{t:RH}, we showed that, in the case $n=2$, $(-1)$-homogeneous solutions to \eqref{eq: HE} are irrotational. 
More generally, in even dimensions $n=2m$, the vorticity two-form of $(-1)$-homogeneous solutions is degenerate, in the sense that $(d u)^{m}=0$. 
In fact, we prove a stronger result for $(-\alpha)$-homogeneous solutions.

\begin{thm}\label{t:W}
Let $n=2m\geq 2$. For $1/m-3<\alpha\leq 1$, the vorticity two-form $du$ of $(-\alpha)$-homogeneous solutions to \eqref{eq: HE} is degenerate.
\end{thm}

\begin{proof}
For $\alpha=1$, it follows from \eqref{eq:V} that $du=d_S v$ for a one-form $v$ on $\mathbb{S}^{2m-1}$. Since $d_Sv$ is a two-form on $\mathbb{S}^{2m-1}$, 
\begin{align*}
(du)^{m}=(d_Sv)^{m}=0.
\end{align*}

For $1/m-3<\alpha<1$, the vorticity function $W=(du)^{m}/\mu$ satisfies \eqref{eq: W}. Arguing as in the proof of Theorem \ref{t:MT}, we obtain
\begin{align*}
\nabla_{S}\cdot (\boldsymbol{v} f W^{k})
= \bigl(k((\alpha+3)m-1) - (2m-1-\alpha)\bigr) f^{2} W^{k}
+ (1-\alpha)|\boldsymbol{v}|^{2} W^{k}
\end{align*}
for $k\in\mathbb{N}$. Choosing $k$ sufficiently large and integrating this identity over $\mathbb{S}^{2m-1}$, we obtain
\[
fW=0 \quad \text{and} \quad \boldsymbol{v}W=0.
\]
If $W\not\equiv 0$, then $f=0$ and $\boldsymbol{v}=0$ on the support of $W$, which implies $du=0$ and hence $W=0$, a contradiction. 
Therefore, $W=0$, and consequently $(du)^{m}=0$. 
\end{proof}

\section{Open questions}\label{S6}

The results obtained in this paper raise a number of further questions. We conclude by discussing several open problems concerning the topology and rigidity of higher-dimensional Euler flows.

\subsection{The four-dimensional Euler equations}

In even dimensions $n=2m$, the vorticity function
\[
W = \frac{(du)^m}{\mu}
\]
is a first integral of the velocity field $\boldsymbol{u}$ of the steady Euler flows in $\mathbb{R}^{2m}$, that is,
\[
\iota_{\boldsymbol{u}} dW = 0.
\]
In particular, in four dimensions, both the Bernoulli function $\Pi$ and the vorticity function $W$ are independent of the first integrals. Consequently, the flow is integrable in the sense that streamlines are constrained to invariant two-dimensional manifolds \cite{GK94}, \cite[Theorem 6.5]{AK21}. In three dimensions, streamlines are integrable on Bernoulli surfaces (i.e., level sets of $\Pi$), whereas in four dimensions, this condition alone does not rule out chaotic behavior. In odd dimensions, streamlines are referred to as \textit{chaotic} when they are not confined to the level sets of a function \cite[Definition 6.1]{AK21}. Moreover, on odd-dimensional Riemannian manifolds, Beltrami-type solutions with chaotic streamlines have been constructed in \cite{Ghrist} through an appropriate choice of the metric; see also \cite[3.4]{C21}.

\begin{q}
Do there exist solutions to the four-dimensional Euler equations whose streamlines exhibit chaotic behavior?
\end{q}

For $(-1)$-homogeneous solutions to the Euler equations on $\mathbb{R}^{2m}\setminus \{0\}$, both the Bernoulli function and the vorticity function vanish; see \S\ref{S4}. In particular, the flow $\boldsymbol{u}=\rho^{-1}\boldsymbol{v}$ provides an example of a four-dimensional solution with $\Pi=0$ and $W=0$. Its streamlines coincide with those of a Beltrami field on $\mathbb{S}^{3}$.

It remains an open problem whether the streamlines of Beltrami fields on $\mathbb{S}^3$ can exhibit chaotic behavior \cite[Example 6.4]{AK21}, in contrast to the ABC flow on $\mathbb{T}^3$. The spectrum of the curl operator on $\mathbb{S}^3$ consists of integers with absolute value greater than or equal to two\cite{Fo89}, and the corresponding eigenfields can realize arbitrary knots and links as streamlines \cite{EPS17}. 

A classical example of a Beltrami field on $\mathbb{S}^3$ is the Hopf field. More generally, the Hopf field on $\mathbb{S}^{2m-1}$ is given by
\begin{align}
\boldsymbol{v}
=
x^2\partial_{x^1}-x^1\partial_{x^2}
+x^4\partial_{x^3}-x^3\partial_{x^4}
+\cdots
+x^{2m}\partial_{x^{2m-1}}-x^{2m-1}\partial_{x^{2m}}. \label{eq: HF}
\end{align}
These vector fields are Beltrami and geodesible on $\mathbb{S}^{2m-1}$. We remark that steady Euler flows on $\mathbb{S}^3$ with non-constant Bernoulli function are constructed in \cite{KKPS14, Sl20}; see also \cite{KKPS20}, \cite[A.II.1]{AK21}.

The Hopf field on $\mathbb{S}^{2m-1}$ has a constant magnitude $|\boldsymbol{v}|=1$ and provides $(-\alpha)$-homogeneous solutions to the Euler equations with non-zero Bernoulli function \eqref{eq:BH} for $\alpha\neq 0,\ 1,\ 2m-1$ and $m\geq 2$ by Theorem \ref{t:MT2}. Its vorticity function $W$ vanishes when $1/m-3<\alpha\leq 1$; see Theorem \ref{t:W}.

\subsection{The rigidity of homogeneous geodesible fields}

Beltrami fields in three dimensions exhibit rigidity. More precisely, there are no nontrivial smooth Beltrami fields satisfying
\[
\boldsymbol{u} = o(|\boldsymbol{x}|^{-1}) \quad \text{as } |\boldsymbol{x}| \to \infty,
\]
see \cite{Na14, CC15}; cf.\ also \cite{EP12, EP15}. Moreover, Beltrami fields do not exist for proportionality factors $\lambda$ whose level sets are diffeomorphic to $\mathbb{S}^2$ (for instance, when $\lambda$ is radially symmetric or possesses a nondegenerate extremum), see \cite{EP16}.

On the other hand, a $(-1)$-homogeneous Euler flow in $\mathbb{R}^3$ satisfies
\[
\boldsymbol{u} = O(|\boldsymbol{x}|^{-1})
\quad \text{as } |\boldsymbol{x}| \to 0 \text{ and } |\boldsymbol{x}| \to \infty,
\]
cf.\cite{Na14, CC15}, and its nonexistence is closely related to that in \cite{EP16}. Indeed,  in two dimensions, the condition 
\[
\iota_{\boldsymbol{v}} d v = 0,
\]
implies
\[
d v = 0,
\]
so that $v$ is a harmonic one-form on $\mathbb{S}^2$ and hence $v=0$. This shows that the nonexistence of $(-1)$-homogeneous Euler flows stems from a reduction to lower-dimensional geometry. This rigidity result extends to $(-\alpha)$-homogeneous Beltrami fields when $\alpha \leq 2$ \cite{Shv}. In contrast, for $\alpha > 2$, there exist $(-\alpha)$-homogeneous Beltrami fields \cite{Abe11}. 

A natural higher-dimensional analogue of a three-dimensional Beltrami field is a geodesible vector field as we showed their non-existence under $(-\alpha)$-homogenity for $1<\alpha\leq n-1$ in Theorem \ref{t:5.1}. A natural question is whether this interval is sharp.

\begin{q}
For $n \geq 4$, do $(-\alpha)$-homogeneous geodesible fields on $\mathbb{R}^n \setminus \{0\}$ exhibit rigidity when $\alpha < 1$, and do they exist when $\alpha > n-1$?
\end{q}

When $n=3$, $(-\alpha)$-homogeneous Beltrami fields do not exist for all $\alpha<1$ \cite{Shv} and its proof relies on the identification of the vorticity two-form with a vector field. Moreover, $(-\alpha)$-homogeneous Beltrami fields for $\alpha > 2$ \cite{Abe11} are obtained by reducing the problem, via axisymmetry, to a semilinear elliptic equation for a stream function. When $n=2$, geodesible fields are irrotational and $(-\alpha)$-homogeneous geodesible fields exist only for integers $\alpha\in \mathbb{Z}$; see Theorem \ref{t:irrotational}.

\subsection{Rigidity under the pointwise bound}

Finally, we turn to the rigidity of steady solutions to the Euler equations that are not $(-1)$-homogeneous. For the Navier--Stokes equations, 
\begin{equation}
\begin{aligned}
\iota_{\boldsymbol{u}}du+d \Pi&=\Delta u,\\
\textrm{div}\ {\boldsymbol{u}}&=0, 
\end{aligned}
\label{eq: NS}
\end{equation}
the scaling invariance is essentially unique and given by
\[
\boldsymbol{u}(\boldsymbol{x})\mapsto \boldsymbol{u}_\lambda(\boldsymbol{x})=\lambda \boldsymbol{u}(\lambda \boldsymbol{x}).
\]
As a consequence, $(-1)$-homogeneous solutions are the only nontrivial homogeneous solutions. 
In contrast to the Euler case, such $(-1)$-homogeneous solutions exist in low dimensions, 
while they do not exist in higher dimensions \cite{Sverak2011}. More precisely, the situation is as follows.

\medskip

\begin{itemize}
\item For $n=2$, homogeneous solutions are known as the Jeffery--Hamel flows, 
which describe radial outflow or inflow in wedge-shaped domains. 
Their existence reduces to a second-order ordinary differential equation with the flux as a parameter; see, e.g., \cite{Fraenkel62}. 
Moreover, the work \cite[Theorem 2]{Sverak2011} establishes a rigidity result in $\mathbb{R}^2\setminus\{0\}$ under the zero-flux condition. 
We also refer to \cite{GW15} for $(-1)$-homogeneous solutions with spiral streamlines and to \cite{Xie3} for the analysis for two-dimensional self-similar solutions in a sector.\\

\item For $n=3$, any $(-1)$-homogeneous solution 
$\boldsymbol{u}\in C^{\infty}(\mathbb{R}^{3}\setminus\{0\})$ is necessarily given by the Landau solutions, 
which are axisymmetric without swirl \cite{TX98, Sverak2011}. See also \cite{JS26} for a recent work.\\
\item For $n\geq 4$, there are no nontrivial $(-1)$-homogeneous solutions 
$\boldsymbol{u}\in C^{\infty}(\mathbb{R}^{n}\setminus\{0\})$ \cite{Tsai1998}, \cite[Theorem 3]{Sverak2011}. The main difference from the Euler equations \eqref{eq: ER} is that, when the Bernoulli function $\Pi$ vanishes, the vorticity $du$ must also vanish.\\ 
\end{itemize}

In dimension $n=3$, it remains an open problem whether the rigidity of Landau solutions 
continues to hold under the weaker condition
\begin{align}
|\boldsymbol{u}(\boldsymbol{x})| \leq \frac{C}{ |\boldsymbol{x}|},  \label{eq:PU}
\end{align}
for arbitrarily large constants $C$; see the recent survey \cite{LiYan}. 
In contrast, in dimensions $n \ge 4$, it has recently been shown by Bang et al. \cite{Xie25} that rigidity holds under the same conditions. See also \cite{Xie2} for the existence of solutions to the forced Navier--Stokes equations satisfying the pointwise estimate \eqref{eq:PU}.

For the Euler equations, one may also consider whether rigidity for $(-1)$-homogeneous solutions holds under the weaker assumption \eqref{eq:PU}. In two dimensions ($n=2$), there exist radially symmetric solutions that are smooth at the origin and satisfy such a pointwise bound; hence, rigidity does not hold under this condition. Similarly, in three dimensions ($n=3$), there exist smooth Beltrami fields that decay at the rate $|\boldsymbol{x}|^{-1}$ \cite{EP12, EP15}, as well as smooth compactly supported solutions \cite{Gav, CLV}. Therefore, rigidity under the pointwise bound \eqref{eq:PU} fails.

\begin{q}
Let $n \geq 4$. Do steady solutions to the Euler equations on $\mathbb{R}^n \setminus \{0\}$ exhibit rigidity under the pointwise bound \eqref{eq:PU}, or do there exist non-geodesible solutions satisfying this bound?
\end{q}

\appendix 

\section{Weitzenb\"{o}ck identity}\label{App}

The system \eqref{eq: HNS} is obtained by adding the spherical and radial components of \(\Delta \boldsymbol{u}\) to the right-hand side of \eqref{eq: HE}, using the following identity.

\begin{lem}\label{l:L}
Let \(\boldsymbol{u}\) be a divergence-free \(( -1)\)-homogeneous vector field of the form \eqref{eq:HV}. Then
\begin{align}
\rho^{3}\Delta \boldsymbol{u}
=
\Delta_{H}\boldsymbol{v}
+
2\nabla_{S}f
+
(\Delta_{S}f)\boldsymbol{y}.   \label{eq:L}
\end{align}
\end{lem}

Using a local orthonormal frame $\{\bold{e}_i\}_{i=1}^{n-1}$ on $\mathbb{S}^{n-1}$, we express the the Laplace-Beltrami operator and the surface divergence on $\mathbb{S}^{n-1}$ as 
\begin{align*}
\Delta_{S}\boldsymbol{v}=\sum_{i=1}^{n-1}D_{\bold{e}_i}D_{\bold{e}_i}\boldsymbol{v}, 
\qquad 
\mathrm{div}_{S}\boldsymbol{v}=\sum_{i=1}^{n-1}\langle \nabla_{\bold{e}_i}\boldsymbol{v},\bold{e}_i\rangle,
\end{align*}
for tangential vector fields $\boldsymbol{v}$ on $\mathbb{S}^{n-1}$, where $D_{\bold{e}_i}\boldsymbol{v}=\bold{e}_i\cdot\nabla \boldsymbol{v}$ denotes the directional derivative along $\bold{e}_i$, and $\langle\cdot,\cdot\rangle$ is the standard inner product on $\mathbb{S}^{n-1}$. By choosing normal coordinates at a point $\boldsymbol{y}\in \mathbb{S}^{n-1}$, i.e., $\nabla_{\bold{e}_i}\bold{e}_i=0$ at $\boldsymbol{y}$, we express the covariant Laplacian on $\mathbb{S}^{n-1}$ as 
\begin{align*}
\Delta_{C}\boldsymbol{v}=\sum_{i=1}^{n-1}\nabla_{\bold{e}_{i}}\nabla_{\bold{e}_{i}}\boldsymbol{v},
\end{align*}
where $\nabla_{\bold{e}_{i}}$ is a covariant derivative along with $\bold{e}_{i}$.

\begin{prop}
The identity 
\begin{align}
\Delta_S \boldsymbol{v}=\Delta_{C}\boldsymbol{v}-\boldsymbol{v}-2(\mathrm{div}_{S}\boldsymbol{v})\,\boldsymbol{y}, \label{eq: A}
\end{align}
holds for tangential vector fields $\boldsymbol{v}$ on $\mathbb{S}^{n-1}$.
\end{prop}

\begin{proof}
Differentiating the Gauss formula
\[
D_{\bold{e}_i}\boldsymbol{v}=\nabla_{\bold{e}_i}\boldsymbol{v}-\langle \boldsymbol{v},\bold{e}_i\rangle \boldsymbol{y},
\]
we find that
\begin{align*}
D_{\bold{e}_i}D_{\bold{e}_i}\boldsymbol{v}
&=\nabla_{\bold{e}_{i}}\nabla_{\bold{e}_{i}}\boldsymbol{v}
-\langle \nabla_{\bold{e}_i}\boldsymbol{v},\bold{e}_i\rangle \boldsymbol{y}
- D_{\bold{e}_i}\bigl(\langle \boldsymbol{v},\bold{e}_i\rangle\bigr)\boldsymbol{y}
-\langle \boldsymbol{v},\bold{e}_i\rangle D_{\bold{e}_i}\boldsymbol{y}.
\end{align*}
Using $D_{\bold{e}_i}\boldsymbol{y}=\bold{e}_i$ and
\begin{align*}
D_{\bold{e}_i}\bigl(\langle \boldsymbol{v},\bold{e}_i\rangle\bigr)
=\langle D_{\bold{e}_i}\boldsymbol{v},\bold{e}_i\rangle
+\langle \boldsymbol{v},D_{\bold{e}_i}\bold{e}_i\rangle
=\langle \nabla_{\bold{e}_i}\boldsymbol{v},\bold{e}_i\rangle,
\end{align*}
we conclude, after summing over $i$, that \eqref{eq: A} holds.
\end{proof}

\begin{prop}
The identities 
\begin{align}
\Delta_{S}(f\boldsymbol{y})&=(\Delta_{S}f )\boldsymbol{y}+2\nabla_S f-(n-1)f\boldsymbol{y}, \label{eq: F1}\\
\Delta_{S}\boldsymbol{v}&=\Delta_{H}\boldsymbol{v}+(n-3)\boldsymbol{v}-2(\textrm{div}_{S}\ \boldsymbol{v})\boldsymbol{y},  \label{eq: F2}
\end{align}
hold for scalar functions $f$ and tangential vector field $\boldsymbol{v}$ on $\mathbb{S}^{n-1}$.
\end{prop}

\begin{proof}
The identity \eqref{eq: F1} follows from a computation using the $\boldsymbol{x}$-coordinates. 
The identity \eqref{eq: F2} follows from \eqref{eq: A} and Weitzenb\"ock identity for the Hodge Laplacian  $\Delta_{C}=\Delta_{H}+n-2$. 
\end{proof}

\begin{proof}[Proof of Lemma \ref{l:L}]
Using the polar coordinates, we find that 
\begin{align*}
\Delta \boldsymbol{u}=\rho^{-3}(\Delta_{S}\boldsymbol{X}+(3-n)\boldsymbol{X}),
\end{align*}
where $\boldsymbol{X}(\boldsymbol{y})=\boldsymbol{v}(\boldsymbol{y})+f(\boldsymbol{y})\boldsymbol{y}$. Applying the identities \eqref{eq: F1} and \eqref{eq: F2} and using the last equation in \eqref{eq: HE}, we obtain \eqref{eq:L}. 
\end{proof}

\bibliographystyle{abbrv}
\bibliography{ref}

\begin{thebibliography}{10}

\bibitem{Abe11}
K.~Abe.
\newblock {E}xistence of homogeneous {E}uler flows of degree $-\alpha\notin [-2,0]$.
\newblock {\em Arch. Rational Mech. Anal.}, 248(30), (2024).

\bibitem{AGJ}
K.~Abe, D.~Ginsberg, and I.-J. Jeong.
\newblock {S}tationary self-similar profiles for the two-dimensional inviscid {B}oussinesq equations.
\newblock {\em Arch. Rational Mech. Anal.}, 250(41), (2026).

\bibitem{AGSJ}
K.~Abe, J.~Gomez-Serrano, and I.-J. Jeong.
\newblock Homogeneous steady states for the generalized surface quasi-geostrophic equations.
\newblock \href{https://arxiv.org/abs/2510.03009}{arXiv:2510.03009}.

\bibitem{AK21}
V.~I. Arnold and B.~A. Khesin.
\newblock {\em Topological methods in hydrodynamics}.
\newblock Springer, Cham, 2021.

\bibitem{AH12}
K.~Atkinson and W.~Han.
\newblock {\em Spherical Harmonics and Approximations on the Unit Sphere: An Introduction}, volume 2044 of {\em Lecture Notes in Mathematics}.
\newblock Springer, Heidelberg, 2012.

\bibitem{Xie2}
J.~Bang, C.~Gui, H.~Liu, Y.~Wang, and C.~Xie.
\newblock {O}n the existence of self-similar solutions to the steady {N}avier--{S}tokes equations in high dimensions.
\newblock \href{https://arxiv.org/abs/2510.10488}{arXiv:2510.10488}.

\bibitem{Xie3}
J.~Bang, C.~Gui, H.~Liu, Y.~Wang, and C.~Xie.
\newblock Self-similar solutions to the steady {N}avier-{S}tokes equations in a two-dimensional sector.
\newblock \href{https://arxiv.org/abs/2412.07283}{arXiv:2412.07283}.

\bibitem{Xie25}
J.~Bang, C.~Gui, H.~Liu, Y.~Wang, and C.~Xie.
\newblock Rigidity of steady solutions to the {Navier--Stokes} equations in high dimensions and its applications.
\newblock {\em J. Eur. Math. Soc. (JEMS)}, (2025).
\newblock Published online first.

\bibitem{C21}
R.~Cardona.
\newblock {S}teady {E}uler flows and {B}eltrami fields in high dimensions.
\newblock {\em Ergodic Theory and Dynamical Systems}, 41(12):3610--3633, (2021).

\bibitem{Car23}
R.~Cardona, E.~Miranda, D.~Peralta-Salas, and F.~Presas.
\newblock Universality of {E}uler flows and flexibility of {R}eeb embeddings.
\newblock {\em Advances in Mathematics}, 428:109142, (2023).

\bibitem{CC15}
D.~Chae and P.~Constantin.
\newblock Remarks on a {L}iouville-type theorem for {B}eltrami flows.
\newblock {\em Int. Math. Res. Not. IMRN}, pages 10012--10016, (2015).

\bibitem{CDG21b}
P.~Constantin, T.~D. Drivas, and D.~Ginsberg.
\newblock Flexibility and rigidity in steady fluid motion.
\newblock {\em Commun. Math. Phys.}, 385:521--563, (2021).

\bibitem{CLV}
P.~Constantin, J.~La, and V.~Vicol.
\newblock Remarks on a paper by {G}avrilov: {G}rad-{S}hafranov equations, steady solutions of the three dimensional incompressible {E}uler equations with compactly supported velocities, and applications.
\newblock {\em Geom. Funct. Anal.}, 29:1773--1793, (2019).

\bibitem{EP12}
A.~Enciso and D.~Peralta-Salas.
\newblock Knots and links in steady solutions of the {E}uler equation.
\newblock {\em Ann. of Math. (2)}, 175:345--367, (2012).

\bibitem{EP15}
A.~Enciso and D.~Peralta-Salas.
\newblock Existence of knotted vortex tubes in steady {E}uler flows.
\newblock {\em Acta Math.}, 214:61--134, (2015).

\bibitem{EP16}
A.~Enciso and D.~Peralta-Salas.
\newblock Beltrami fields with a nonconstant proportionality factor are rare.
\newblock {\em Arch. Ration. Mech. Anal.}, 220:243--260, (2016).

\bibitem{EPS17}
A.~Enciso, D.~Peralta-Salas, and F.~Torres~de Lizaur.
\newblock Knotted structures in high-energy {B}eltrami fields on the torus and the sphere.
\newblock {\em Ann. Sci. \'{E}c. Norm. Sup\'{e}r. (4)}, 50(4):995--1016, (2017).

\bibitem{Fo89}
G.~B. Folland.
\newblock Harmonic analysis of the de rham complex on the sphere.
\newblock {\em Journal f{\"u}r die reine und angewandte Mathematik}, 398:130--143, (1989).

\bibitem{Fraenkel62}
L.~E. Fraenkel.
\newblock Laminar flow in symmetrical channels with slightly curved walls. {I}. {O}n the {J}effery-{H}amel solutions for flow between plane walls.
\newblock {\em Proc. Roy. Soc. London Ser. A}, 267:119--138, (1962).

\bibitem{Gav}
A.~V. Gavrilov.
\newblock A steady {E}uler flow with compact support.
\newblock {\em Geom. Funct. Anal.}, 29:190--197, (2019).

\bibitem{Ghrist}
R.~Ghrist.
\newblock Steady nonintegrable high-dimensional fluids.
\newblock {\em Lett. Math. Phys.}, 55(no. 3):pp. 193--204, (2001).

\bibitem{GK94}
V.~L. Ginzburg and B.~A. Khesin.
\newblock Steady fluid flows and symplectic geometry.
\newblock {\em J. Geometry and Physics}, 14(no. 2):195--210, (1994).

\bibitem{Gluck}
H.~Gluck.
\newblock Open letter on geodesible flows.
\newblock Unpublished manuscript (cited in Sullivan (1978)), 1970s.

\bibitem{GPMPS25}
{\'A}.~Gonz{\'a}lez-Prieto, E.~Miranda, and D.~Peralta-Salas.
\newblock Universality in computable dynamical systems: old and new.
\newblock {\em Journal of Physics: Complexity}, 6:035014, (2025).

\bibitem{GW15}
J.~Guillod and P.~Wittwer.
\newblock Generalized scale-invariant solutions to the two-dimensional stationary {N}avier-{S}tokes equations.
\newblock {\em SIAM J. Math. Anal.}, 47(1):955--968, (2015).

\bibitem{JS26}
H.~Jia and V.~Sverak.
\newblock {R}efined asymptotics of the steady {N}avier {S}tokes equation around small {L}andau solutions.
\newblock \href{https://arxiv.org/abs/2605.24200}{arXiv:2605.24200}.

\bibitem{KKPS14}
B.~Khesin, S.~Kuksin, and D.~Peralta-Salas.
\newblock {K}{A}{M} theory and the 3{D} {E}uler equation.
\newblock {\em Advances in Mathematics}, 267:498--522, (2014).

\bibitem{KKPS20}
B.~Khesin, S.~Kuksin, and D.~Peralta-Salas.
\newblock Global, local and dense non-mixing of the {3D} {E}uler equation.
\newblock {\em Arch. Rational Mech. Anal.}, 238(2):1087--1112, (2020).

\bibitem{LiYan}
L.~Li and X.~Yan.
\newblock Recent research on $(-1)$-homogeneous solutions of stationary {N}avier--{S}tokes equations.
\newblock \href{https://arxiv.org/abs/2509.07243}{arXiv:2509.07243}.

\bibitem{LuoShvydkoy}
X.~Luo and R.~Shvydkoy.
\newblock 2{D} homogeneous solutions to the {E}uler equation.
\newblock {\em Comm. Partial Differential Equations}, 40:1666--1687, (2015).

\bibitem{LS17}
X.~Luo and R.~Shvydkoy.
\newblock Addendum: 2{D} homogeneous solutions to the {E}uler equation.
\newblock {\em Comm. Partial Differential Equations}, 42(3):491--493, (2017).

\bibitem{Na14}
N.~Nadirashvili.
\newblock Liouville theorem for {B}eltrami flow.
\newblock {\em Geom. Funct. Anal.}, 24:916--921, (2014).

\bibitem{PC}
M.~M.~G. Pascual-Caballo.
\newblock Stationary radial homogeneous solutions for the inviscid {SQG} equation.
\newblock \href{https://arxiv.org/abs/2510.03108}{arXiv:2510.03108}.

\bibitem{PSRT21}
D.~Peralta-Salas, A.~Rechtman, and F.~Torres~de Lizaur.
\newblock A characterization of 3d steady euler flows using commuting zero-flux homologies.
\newblock {\em Ergodic Theory and Dynamical Systems}, 41(7):2166--2181, (2021).

\bibitem{Shv}
R.~Shvydkoy.
\newblock Homogeneous solutions to the 3{D} {E}uler system.
\newblock {\em Trans. Amer. Math. Soc.}, 370:2517--2535, (2018).

\bibitem{Sl20}
R.~Slobodeanu.
\newblock Steady {E}uler flows on the 3-sphere and other {S}asakian 3-manifolds.
\newblock {\em Ann. Global Anal. Geom.}, 58(4):561--575, (2020).

\bibitem{Sohr}
H.~Sohr.
\newblock {\em The {N}avier-{S}tokes equations}.
\newblock Birkh\"auser Advanced Texts: Basler Lehrb\"ucher. Birkh\"auser Verlag, Basel, 2001.

\bibitem{Tao18}
T.~Tao.
\newblock On the universality of the incompressible euler equation on compact manifolds.
\newblock {\em Discrete Contin. Dyn. Syst.}, 38(3):1553--1565, (2018).

\bibitem{Tao20}
T.~Tao.
\newblock On the universality of the incompressible {E}uler equation on compact manifolds, {II}. {N}on-rigidity of {E}uler flows.
\newblock {\em Pure Appl. Funct. Anal.}, 5(6):1425--1443, (2020).

\bibitem{TX98}
G.~Tian and Z.~Xin.
\newblock One-point singular solutions to the {N}avier-{S}tokes equations.
\newblock {\em Topol. Methods Nonlinear Anal.}, 11:135--145, (1998).

\bibitem{Tsai1998}
T.-P. Tsai.
\newblock {\em On Problems Arising in the Regularity Theory for the Navier--Stokes Equations}.
\newblock PhD thesis, University of Minnesota, 1998.
\newblock MR2697733.

\bibitem{Sverak2011}
V.~\v{S}ver\'{a}k.
\newblock On {L}andau's solutions of the {N}avier-{S}tokes equations.
\newblock {\em J. Math. Sci. (N.Y.)}, 179:208--228, (2011).

\end{thebibliography}

\end{document}